\documentclass{amsart}

\usepackage{amsmath,amssymb,amsfonts,enumerate,amsthm}

\newcommand{\Spec}{\mbox{Spec}\,}

\newcommand{\depth}{\mbox{depth}\,}

\renewcommand{\dim}{\mbox{dim}\,}

\newcommand{\E}{\mbox{E}}

\newcommand{\V}{\mbox{V}}

\newcommand{\fp}{\mathfrak{p}}

\newtheorem{thm}{Theorem}[section]
\newtheorem{cor}[thm]{Corollary}
\newtheorem{lem}[thm]{Lemma}
\newtheorem{prop}[thm]{Proposition}

\newtheorem{rem}[thm]{Remark}

\numberwithin{equation}{section}

\begin{document}
\bibliographystyle{amsplain}

%\date{}
\title{Bipartite $S_2$ graphs are Cohen-Macaulay}

\author{Hassan Haghighi}

\address{Hassan Haghighi\\Department of Mathematics, K. N. Toosi
University of Technology, Tehran, Iran.}

\author{Siamak Yassemi}
\address{Siamak Yassemi\\Department of Mathematics, University of
Tehran, Tehran, Iran and School of Mathematics, Institute for
Research in Fundamental Sciences (IPM), Tehran Iran.}

\author{Rahim Zaare-Nahandi}
\address{Rahim Zaare-Nahandi\\School of Mathematics, Statistics \&
Computer Science, University of Tehran, Tehran, Iran.}
%\thanks{H. Haghighi was supported in part by a grant from...}
%\thanks{S. Yassemi was supported in part by a grant from...}

\thanks{Siamak Yassemi was supported in part by a grant from IPM No. 870130211}
\thanks{Emails: haghighi@kntu.ac.ir, yassemi@ipm.ir, rahimzn@ut.ac.ir}



\begin{abstract}

\noindent In this paper we show that if the Stanley-Reisner ring of
the simplicial complex of independent sets of a bipartite graph $G$
satisfies Serre's condition $S_2$, then $G$ is Cohen-Macaulay. As a
consequence, the characterization of Cohen-Macaulay bipartite graphs
due to Herzog and Hibi carries over this family of bipartite graphs.
We check that the equivalence of Cohen-Macaulay property and the
condition $S_2$ is also true for chordal graphs and we classify
cyclic graphs with respect to the condition $S_2$.
\end{abstract}

\maketitle

\section*{Introduction}

Let $k$ be a field. To any finite simple graph $G$ with vertex set
$\V=[n]=\{1,\cdots,n\}$ and edge set $\E(G)$ one associates an ideal
$I(G)\subset k[x_1,\cdots ,x_n]$ generated by all monomials $x_ix_j$
such that $\{i,j\}\in\E(G)$. The ideal $I(G)$ and the quotient ring
$k[x_1,\cdots ,x_n]/I(G)$ are called the edge ideal of $G$ and the
edge ring of $G$, respectively. The simplicial complex of $G$ is
defined by
$$\Delta_G=\{A\subseteq \V|A \,\, \mbox{is an independent set in}\,\,
G\},$$ where $A$ is an independent set in $G$ if none of its
elements are adjacent. Note that $\Delta_G$ is precisely the
simplicial complex with the Stanley-Reisner ideal $I(G)$.

A graph $G$ is said to be Cohen-Macaulay (resp. Buchsbaum) over $k$,
if the edge ring of $G$ $k[x_1,\cdots,x_n]/I(G)$ is Cohen-Macaulay
(resp. Buchsbaum), and is called Cohen-Macaulay (resp. Buchsbaum) if
it is Cohen-Macaulay (resp. Buchsbaum) over any field. A graph is
said to be chordal if each cycle of length $> 3$ has a chord.

Let $\Delta$ be a simplicial complex. This complex is called
disconnected if the vertex set  $V$ of $\Delta$ is the disjoint
union of two nonempty sets $V_1$ and $V_2$ such that no face of
$\Delta$ has vertices in both $V_1$ and $V_2$, otherwise it is
called connected. A simplicial complex $\Delta$ is called
Cohen-Macaulay (resp. Buchsbaum) over an infinite field $k$ if its
Stanley-Reisner ring $k[\Delta]$ is Cohen-Macaulay (resp.
Buchsbaum).

It is known that if $\Delta$ is a disconnected simplicial complex,
then $\depth k[\Delta]=1$, \cite[Chapter 5, Ex. 5.1.26]{BH}. This
implies that if $\depth k[\Delta]>1$, then $\Delta$ is connected. In
particular, every Cohen-Macaulay simplicial complex of positive
dimension is connected.

A satisfactory classification of all Cohen-Macaulay graphs over a
field $k$ has been standing open for some time. However, as pointed
out by Herzog et al \cite[Introduction]{HHZ}, this is equivalent to
a classification of all Cohen-Macaulay simplicial complexes over $k$
which is clearly a hard problem. Accordingly, it is natural to study
special families of Cohen-Macaulay graphs. Recall that a graph $G$
on the vertex set $[n]$ is bipartite if there exists a partition
$[n]=V\cup W$ with $V\cap W=\varnothing$ such that each edge of $G$
is of the form $\{i,j\}$ with $i\in V$ and $j\in W$. It is easy to
see that a graph $G$ is bipartite if and only if it has no cycle of
odd length. For a Cohen-Macaulay bipartite graph $G$, Estrada and
Villarreal \cite{EV} showed that $G\setminus\{\nu\}$ is
Cohen-Macaulay for some vertex $\nu\in\V(G)$. In \cite{V1} it is
shown that the cyclic graph $C_n$ is Cohen-Macaulay if and only if
$n\in \{3,5\}$. Herzog and Hibi gave a graph-theoretic
characterization of all bipartite Cohen-Macaulay graphs. Due to our
direct application, we state their result.

\vspace{.1in}

\noindent{\bf Theorem} \cite[Theorem 3.4]{HH}. Let $G$ be a
bipartite graph with vertex partition $V\cup W$. Then the following
conditions are equivalent:

\begin{itemize}

\item[(a)] $G$ is a Cohen-Macaulay graph;

\item[(b)] $|V|=|W|$ and the vertices $V=\{x_1,\cdots,x_n\}$ and
$W=\{y_1,\cdots, y_n\}$ can be labeled such that:

\begin{verse}

(i) $\{x_i,y_i\}$ are edges for $i=1, \cdots , n$;

(ii) if $\{x_i,y_j\}$ is an edge, then $i\le j$;

(iii) if $\{x_i,y_j\}$ and $\{x_j,y_k\}$ are edges, then
$\{x_i,y_k\}$ is also an edge.

\end{verse}

\end{itemize}

Note that this result is characteristic-free.

Let $G$ be a graph with vertex set $\V(G)$ and edge set $\E(G)$. A
subset $C\subseteq\V(G)$ is a {\it minimal vertex cover} of $G$ if:
(1) every edge of $G$ is incident with a vertex in $C$, and (2)
there is no proper subset of $C$ with the first property. Observe
that a minimal vertex cover is the set of indeterminates which
generate a minimal prime ideal in the prime decomposition of $I(G)$.
Also note that $C$ is a minimal vertex cover if and only if
$\V(G)\setminus C$ is a maximal independent set, i.e., a facet of
$\Delta_G$.

A graph $G$ is called {\it unmixed} if all minimal vertex covers of
$G$ have the same number of elements, i.e., $\Delta_G$ is pure. It
is well known that every Cohen-Macaulay graph $G$ is unmixed. A
graph is called chordal if every cycle of length $>3$ has a chord.
Recall that a chord of a cycle is an edge which joins two vertices
of the cycle but is not itself an edge of the cycle.

Recall that a finitely generated graded module $M$ over a Noetherian
graded $k$-algebra $R$ is said to satisfy the Serre's condition
$S_n$ if
\[ \depth M_{\mathfrak p} \ge \min(n, \dim M_{\fp}), \]
for all ${\fp} \in \Spec(R).$ Thus, $M$ is Cohen-Macaulay if and
only if it satisfies the Serre's condition $S_n$ for all $n$. A
graph is said to satisfy the Serre's condition $S_n$, or simply is
an $S_n$ graph, if its edge ring satisfies this condition. Using
\cite[Lemma 3.2.1]{Sc} and Hochster's formula on local cohomology
modules, a pure $d$-dimensional Stanley-Reisner ring $k[\Delta]$
satisfies $S_2$ property if and only if
$\widetilde{H}_0(\mathrm{link}_\Delta(F);k)=0$ for all $F\in \Delta$
with $|F| \le d-2$ (see \cite[page 4]{T}).

The main result of this paper is to prove that if $G$ is a bipartite
$S_2$ graph, then $G$ is Cohen-Macaulay (see Theorem 1.3).
Consequently, the characterization of Cohen-Macaulay bipartite
graphs by Herzog and Hibi carries over bipartite $S_2$ graphs. It is
shown that not only for bipartite graphs but also for chordal graphs
Cohen-Macaulay property and the condition $S_2$ are equivalent. To
see an example of a non-Cohen-Macaulay $S_2$ graph, it is shown that
the cyclic graph $C_n$ of length $n\ge 3$ is $S_2$ if and only if
$n=3,5$ or $7$. In particular, $C_7$ is the only cyclic graph which
is $S_2$ but not Cohen-Macaulay. Finally, we reprove some known
results on certain bipartite Cohen-Macaulay graphs by providing
rather simpler proofs compared to the existing ones.

\section{The Main Result}

Our results are inspired by the aforementioned theorem of Herzog and
Hibi \cite[Theorem 3.4]{HH}.

\begin{prop}
Let $G$ be an unmixed bipartite graph with bipartition $V=\{x_1,
\cdots ,x_n\}$ and $W=\{y_1, \cdots ,y_n\}$ such that $\{x_i,y_i\}$
is an edge of $G$ for all $i=1,\cdots ,n$. Then $V$ and $W$ can be
simultaneously relabeled such that the following statements are
equivalent:

\begin{itemize}

\item[(a)] There exists a linear order $V=F_0,...,F_n=W$ on some of the
facets of $\Delta_G$  such that $F_i$ and $F_{i+1}$ intersect in
codimension one for $i=0, \cdots ,n-1$.

\item[(b)] If $\{x_i,y_j\}$ is an edge, then $i \le j$.

\end{itemize}
\end{prop}

By a simultaneous relabeling we mean that for all $i$, $x_i$ and
$y_i$ receive the same relabeling. In particular, under the
assumptions of Proposition 1.1, with the new labeling, $\{x_i,y_i\}$
is an edge of $G$ for all $i=1,\cdots ,n$.

Before proceeding on the proof of this Proposition note that the
condition (a) is weaker than strongly connectedness of $\Delta_G$.
Recall that a simplicial complex $\Delta$ is strongly connected if
for any two facets  $V$ and $W$ of $\Delta$ there exists a chain of
facets satisfying (a). Here we only need this sequence just for the
two specific facets $V$ and $W$.

\begin{proof} (a)$\Rightarrow$(b): We have $|F_1\setminus F_0|=1$,
say $F_1\setminus F_0=\{y_1\}$. Then $F_1=\{y_1,x_2,\cdots ,x_n\}$
because $\{x_1,y_1\}$ is not a face of $\Delta_G$. Similarly,
$|F_2\setminus F_1|=1$, say $F_2\setminus F_1=\{y_2\}$. Thus
$F_2=\{y_1,y_2,x_3,\cdots,x_n\}$ because again $\{x_2,y_2\}$ is not
a face of $\Delta_G$. Hence by induction we may assume that
$F_i=\{y_1,\cdots,y_i,x_{i+1},\cdots,x _n\}$ for $i=0,\cdots,n$. In
particular, if $i>j$, then $\{x_i,y_j\}$ is a face of $\Delta_G$,
and hence it is not an edge of $G$.

(b)$\Rightarrow$(a): Set
$F_i=\{y_1,\cdots,y_i,x_{i+1},\cdots,x_n\}$. It is easy to see that
for any $i$, $F_i$ is a maximal independent set and hence a facet of
$\Delta_G$. Moreover $F_i$ and $F_{i+1}$ intersect in codimension
one.

\end{proof}

\begin{lem}
Let $G$ be a bipartite graph. Then $G$ is a non-complete bipartite
graph if and only if $\Delta_G$ is connected.
\end{lem}

\begin{proof}
Let $V_1 \cup V_2$ be the bipartition of $G$. Then $G$ fails to be a
complete bipartite graph if and only if there are two vertices $
x\in V_1$ and $y \in V_2$ which are not adjacent, that is, $\{x,y\}$
is an independent set of $G$, i.e., $\Delta_G$ is connected.
\end{proof}

Now we may state the main result which in particular provides a
characterization of bipartite $S_2$ graphs.

\begin{thm}
Let $G$ be a bipartite graph with at least four vertices and with
vertex partition $V$ and $W$. Then the following are equivalent:

\begin{itemize}

\item[(a)] $G$ is unmixed and $V$ and $W$ can be labeled such that there
exists an order $V=F_0,\cdots,F_n=W$ of the facets of $\Delta_G$
where $F_i$ and $F_{i+1}$ intersect in codimension one for
$i=0,\cdots,n-1$.

\item[(b)] $G$ is a Cohen-Macaulay graph.

\item[(c)] $G$ is a Buchsbaum non-complete bipartite
graph.
\item[(d)] $G$ is an $S_2$ graph.

\end{itemize}
\end{thm}

\begin{proof}
We prove $ (a) \Rightarrow (b) \Rightarrow (c) \Rightarrow (d)
\Rightarrow (a).$

(a)$\Rightarrow$(b): Since $G$ is unmixed, by K\"{o}nig's Theorem
there is a bipartition $V=\{x_1,\cdots,x_n\}$ and
$W=\{y_1,\cdots,y_n\}$ such that $\{x_i,y_i\}$ is an edge of $G$ for
all $i$. By Proposition 1.1, $V$ and $W$ can be relabeled such that
$\{x_i,y_i\}$ is an edge of $G$ for all $i$ and if $\{x_i,y_j\}$ is
an edge in $G$, then $i\le j$. We fix such a labeling. Let
$\{x_i,y_j\}$ and $\{x_j,y_k\}$ be edges of $G$ with $i<j<k$, and
suppose that $\{x_i,y_k\}$ is not an edge of $G$. Since
$\{x_i,y_k\}$ is a face of $\Delta_G$ and  $G$ is unmixed,
$\Delta_G$ is pure, hence there exists a facet $F$ of $\Delta_G$
with $|F|=n$ and $\{x_i,y_k\}\subset F$. Since $F$ is a facet of
$\Delta_G$, any 2-element subset of $F$ is a non-edge of $G$. We
have $y_j \notin F$ since $\{x_i,y_j\}$ is an edge of $G$. Similarly
$x_j \not \in F$ since $\{x_j,y_k\}$ is an edge of $G$. On the other
hand, since $\{x_t,y_t\}$ is an edge of $G$ for all $t$, the facet
$F$ can not contain both $x_t$ and $y_t$. Hence $F$ is of the form
$F=\{z_1,\cdots,z_n\}$, where $z_t=x_t$ or $y_t$ for $t=1,\cdots,n$.
Thus either $y_j$ or $x_j$ belongs to $F$, which is a contradiction.
Consequently, $G$ is Cohen-Macaulay by the theorem of Herzog and
Hibi.

(b)$\Rightarrow$(c): Since every Cohen-Macaulay ring is a Buchsbaum
ring, $G$ is also Buchsbaum. By definition, the ideal of the
simplicial complex $\Delta_G$ is equal to edge ideal of $G$. Hence
$\Delta_G$ is also Cohen-Macaulay and in particular, $\Delta_G$ is
connected. Therefore, by Lemma 1.2 $G$ is non-complete.

(c)$\Rightarrow$(d): By \cite[Corollary 2.7]{Y} the localization of
every Buchsbaum ring at any of its prime ideals which is not equal
to $(x_1,\cdots,x_n,y_1,\cdots,y_n)$, is Cohen-Macaulay. Therefore
$G$ satisfies the $S_2$ condition.

(d)$\Rightarrow$(a): Since $\Delta_G$ satisfies the $S_2$ condition,
by \cite[Corollary 2.4]{Hart} for any two facets $F$ and $H$ of
$\Delta_G$, there exist a positive integer $m$ and a sequence
$F=F_0,\cdots, F_m=H$ of facets of $\Delta_G$ such that $F_{i}$
intersects $F_{i+1}$ in codimension one for all $i=0,\cdots ,m-1$.
Hence $\Delta_G$ is strongly connected. In particular, since the
partitions $V$ and $W$ of the vertices of $G$ can be considered as
two facets of $\Delta_G$ and $\Delta_G$ is strongly connected, the
required sequence exists. Furthermore, $\mid F_i\mid=\mid F_i \cap
F_{i+1} \mid + 1= \mid F_{i+1} \mid$ for all $i=0,\cdots ,m-1$. This
implies that any two facets of $\Delta_G$ have the same number of
elements and hence $G$ is unmixed.
\end{proof}

\begin{rem}{\rm
The implication (b)$\Rightarrow$(a) in the above theorem does not
depend on the bipartite assumption of $G$ and is valid in a more
general setting. In fact a stronger implication is valid. More
precisely, every Cohen-Macaulay simplicial complex is strongly
connected. This follows, for example, by an argument similar to the
implication (d)$\Rightarrow$(a)}.
\end{rem}

\begin{rem}{\rm
Theorem 1.3 reveals that for bipartite graphs Cohen-Macaulay and
$S_2$ properties are equivalent. This raises the question whether
there are other families of graphs for which these two properties
are equivalent. Here, we show that,

\begin{itemize}

\item[(1)] Every chordal $S_2$ graph is Cohen-Macaulay.

\item[(2)] The cyclic graph $C_7$ is $S_2$ but not Cohen-Macaulay.

\end{itemize}

In fact, chordal graphs are shellable \cite[Theorem 2.13]{VV}. But
any $S_2$ graph is unmixed (see \cite[Corollary 5.10.9]{Gr}, or
\cite[Remark 2.4.1]{Hart}). Therefore, for chordal graphs
Cohen-Macaulay and $S_2$ properties are equivalent.

To establish (2) we classify all cyclic graphs $C_n$ with respect to
$S_2$ property.}

\end{rem}

\begin{prop}
The cyclic graph $C_n$ of length $n\ge 3$ is $S_2$ if and only if
$n=3,5$ or $7$. In particular, $C_7$ is the only cyclic graph which
is $S_2$ but not Cohen-Macaulay.
\end{prop}
\begin{proof} It is known that $C_n$ is Cohen-Macaulay if and only
if $n=3, 5$ \cite[Corollary 6.3.6]{V1}. On the other hand, $C_n$ of
length $n\ge 3$ is unmixed if and only if $n=3, 4, 5, 7$
\cite[Exercise 6.2.15]{V1}. Accordingly, $C_3$ and $C_5$ are $S_2$.
Since $C_4$ is bipartite but not Cohen-Macaulay, by Theorem 1.3 it
is not $S_2$. Furthermore, as mentioned before, every $S_2$ graph is
unmixed. Thus, the only cyclic graph which remains to be checked is
$G=C_7$. To settle this, we apply the cohomological criterion for
$S_2$ property mentioned in the introduction. In fact, we need to
check that for all $F\in \Delta_G$ with $|F| \le 1$,
$\widetilde{H}_0(\mathrm{link}_{\Delta_G}(F);k)=0$. This condition
is satisfied if $\mathrm{link}_{\Delta_G}(F)$ is connected which can
easily be checked by direct inspection.
\end{proof}

In light of Theorem 1.3, we consider some known results on certain
bipartite Cohen-Macaulay graphs and we provide rather simpler proofs
compared to the existing ones.

As a consequence of Theorem 1.3(b) we may state the following result
on the structure of trees satisfying the condition $S_2$.

\begin{cor} \cite[Theorem 6.3.4]{V1}
Let $G$ be a tree with at least four vertices. Then the following
are equivalent:

\begin{itemize}

\item[(a)] $G$ satisfies the condition $S_2$.

\item[(b)] There is a bipartition $V=\{x_1,\cdots, x_n\}$,
$W=\{y_1,\cdots ,y_n\}$ of $G$ such that

\begin{verse}

(i) $\{x_i,y_i\}\in\E(G)$ for all $i$.

(ii) for each $i$ either $\deg(x_i)=1$ or $\deg(y_i)=1$,
exclusively.

(iii) $V$ and $W$ can be simultaneously relabeled such that there
exists an order $V=F_0,\cdots ,F_n=W$ of the facets of $\Delta_G$
where $F_i$ and $F_{i+1}$ intersect in codimension one for
$i=0,\cdots ,n-1$.

\end{verse}

\end{itemize}

\end{cor}

From part (b)$(ii)$ of Corollary 1.7 it follows that every tree with
$2n$ vertices which satisfies the condition $S_2$, has precisely $n$
vertices of degree one.

\begin{cor}
Every path of length greater than four does not satisfy the
condition $S_2$ and hence it is not Cohen-Macaulay.
\end{cor}

By Corollary 1.7 every bipartite $S_2$ graph has at least two
vertices of degree one. From this fact and Theorem 1.3 we get the
following result which is a special case of \cite[Proposition
6.2.1]{V1}.

\begin{prop}
Let $G$ be a bipartite $S_2$ graph. Let $y$ be a vertex of degree
one of $G$ and $x$ its adjacent vertex. Then $G \setminus \{x,y\}$
is still an $S_2$ graph.
\end{prop}

\begin{proof} Since $G$ is bipartite, there exists an order
$V=F_0,\cdots, F_n=W$ of facets of $\Delta_G$ such that for each $i
=0,\cdots, n-1, F_i$ intersects $F_{i+1}$ in codimension one. Since
for each $i$, $V \cup W \setminus F_i$ is a minimal vertex cover of
$G$, it contains exactly one of the vertices $x$ or $y$. Thus $F_i$
contains $y$ or $x$ respectively. Again since any facet of
$\Delta_G$ is an independent set, none of these facets can contain
both of these elements. Thus, if we delete both of these elements
from $V(G)$, then they will be deleted from each element of the
sequence $V=F_0, \cdots, F_n=W$. By construction
$F_0\setminus\{x\}=F_1\setminus\{y\}$, and hence we obtain a
sequence of length $n-1$ of facets of $\Delta_{G\setminus \{x,y\}}$
such that each two consecutive members of this sequence intersect
each other in codimension one. Now the claim follows from Theorem
1.3(b).

\end{proof}

\begin{rem}{\rm
A careful inspection of the proof of Proposition 1.9 reveals that
every edge $\{x,y\}$ where $y$ is an arbitrary degree one vertex of
$G$, intersects every member of the sequence $F_0,\cdots,F_n$.
Conversely, if we add a new vertex $x_{n+1}$ to $V$ and a new vertex
$y_{n+1}$ to $W$ and the edge $\{x_{n+1},y_{n+1}\}$ to $G$, then the
bipartite graph $G_1=V_1\cup W_1$, where $V_1=V\cup \{x_{n+1}\}$ and
$W_1=W \cup \{y_{n+1}\}$, has the sequence $F_0\cup \{x_{n+1}\},
F_1\cup \{x_{n+1}\},\cdots, F_n\cup \{x_{n+1}\}, F_{n+1}=F_n \cup
\{y_{n+1}\}$ as a subsequence of its facets which satisfies the
assumption of Theorem 1.3(b), hence $G_1$ is an $S_2$ graph.}
\end{rem}

We end this paper with the following immediate result which is again
a special case of \cite[Proposition 6.2.1]{V1}.

\begin{cor}
Let $G$ be a tree with more than two vertices which is $S_2$. Let
$x$ be a degree one vertex of $G$ and $y$ its adjacent vertex. Then
$G \setminus \{x,y\}$ is an $S_2$ graph.
\end{cor}

\section*{Acknowledgments}

This paper was initiated during Yassemi's visit of the Max-Planck
Institute fur Mathematics (MPIM). He would like to thank the
authorities of MPIM for their hospitality during his stay there. The
authors are grateful to the referee for a careful review of this
paper and for several corrections and remarks.

\providecommand{\bysame}{\leavevmode\hbox
to3em{\hrulefill}\thinspace}

\end{document}